\documentclass{article}

\usepackage{amsmath,amssymb,amsthm}
\usepackage{geometry}                
\usepackage{diagram}
\usepackage{enumerate}
\usepackage{graphicx}
\usepackage{subfig}

\theoremstyle{plain}
\numberwithin{equation}{section}
\newtheorem{thm}{Theorem}[section]
\newtheorem{theorem}[thm]{Theorem}

\newtheorem{example}[thm]{Example}
\newtheorem{definition}[thm]{Definition}
\newtheorem{proposition}[thm]{Proposition}
\newtheorem{corollary}[thm]{Corollary}

\title{The Problem of Pawns}

\author{Tricia Muldoon Brown\\
Georgia Southern University}

\date{}

\begin{document}

\maketitle

\begin{abstract}
Using a bijective proof, we show the number of ways to arrange a maximum number of nonattacking pawns on a $2m\times 2m$ chessboard is ${2m\choose m}^2$, and more generally, the number of ways to arrange a maximum number of nonattacking pawns on a $2n \times 2m$ chessboard is ${m+n\choose n}^2$.
\end{abstract}

\setboolean{piececounter}{false}
\setboolean{showcomputer}{false}

\section{Introduction}
A set of pieces on a chessboard is said to be independent if no piece may attack another.  Independence problems on chessboards have long been studied; both in terms of maximum arrangements as well as the number of such arrangements.  For all traditional chess pieces, kings, queens, bishops, rooks, knights, and pawns, the maximum size of an independent set is known.  When enumerating maximum arrangements, some of the problems, for example in the case of rooks or bishops, have elementary solutions.  (See Dudeney~\cite{Dudeney} for an early discussion of independence problems.)  For other pieces, such as in the case of queens, the number of maximum independent arrangements is unknown, or in the case of kings an asymptotic approximation is given by Larson~\cite{Larson}, but an exact value is unknown.  Here we wish to enumerate the number of maximum arrangements of nonattacking pawns.  Arrangements of nonattacking pawns have been studied by Kitaev and Mansour~\cite{Kitaev_Mansour} who provide upper and lower bounds on the number of arrangements of pawns on $m\times n$ rectangles using Fibonacci numbers as well as an algorithm to generate an explicit formula.

As there are only two distinct arrangements for odd length chessboards, we focus on boards with even length.  Because we can divide a $2m\times 2m$ chessboard into $m^2$ $2\times 2$ squares each with at most two pawns, the maximum number of independent pawns is at most $2m^2$.  This value is easily achieved, and examples are illustrated in Figure~\ref{pawns}.
\begin{figure}[ht]
\centering
\includegraphics{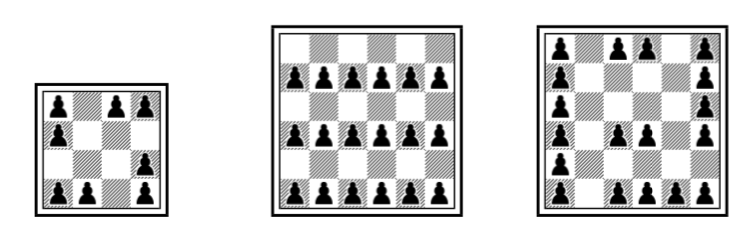}
\caption{Arrangements of nonattacking pawns for even length chessboards}
\label{pawns}
\end{figure}

We will provide a bijection between the set of maximum nonattacking arrangements of pawns on a $2m\times 2m$ chessboard and the set of subsets of $m$ rows and $m$ columns of a $2m\times 2m $ matrix.

\section{Bijection}
Instead of considering full arrangements of nonattacking pawns on a $2m\times 2m$ chessboard, we first consider arrangements on a $2\times 2$ chessboard.  There are four possible arrangements labeled with A, B, C, and D, as illustrated in Figure~\ref{2times2}.  We define a function $f$ on this set, where $f(A) = D$ and $f(B)=f(C)=f(D)=C$.  We use this function to define an $(m+1)\times (m+1)$ matrix $M_{2m} = (m_{i,j})_{1\leq i, j \leq m+1}$ whose entries correspond to arrangements of $2m$ independent pawns on a $2 \times 2m$ rectangular chessboard.
   
\begin{figure}[ht]
\centering
\includegraphics{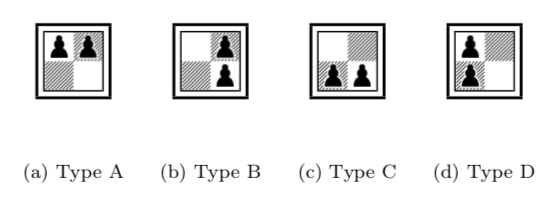}
\caption{The four maximum arrangements of 2 pawns on a $2\times 2$ chessboard}
\label{2times2}
\end{figure}

\begin{definition}\label{defM2m}
Let $M_{2m}=(m_{i,j})_{1\leq i,j \leq m+1}$ be the matrix who entries consist of arrangements of $2m$ nonattacking pawns on a $2\times 2m$ rectangular chessboard. We can think of each rectangle as a string of $m$ $2\times 2$ squares, each with exactly two pawns.  The entries of $M_{2m}$ are defined as follows:
\begin{enumerate}[i.]
\item For $1\leq j \leq m+1$, let $m_{1,j}$ be the arrangement where the leftmost $(m+1-j)$ $2\times 2$ squares of the rectangular chessboard are of Type A and the remaining rightmost $(j-1)$ squares are of Type B.   
\item For $1\leq i \leq m+1$, use $m_{1,j}$ to generate the arrangements $m_{i,j}$ by replacing the leftmost $(i-1)$ $2\times 2$ squares of $m_{1,j}$, with their image under the function $f$ and leaving the rightmost $(m+1-i)$ $2\times 2$ squares fixed. 
\end{enumerate}
\end{definition}

See Figure~\ref{entries} for an example of an entry in the first row and fifth row of $M_{14}$, and see Figure~\ref{M6} for the entire matrix $M_6$.  We claim this matrix contains all possible nonattacking arrangements of pawns on a $2\times 2m$ rectangular chessboard.
 
\begin{figure}[ht]
\centering
\includegraphics{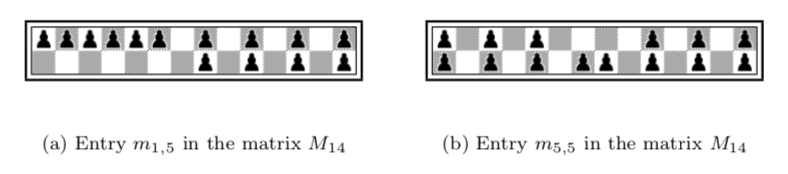}
\caption{Entries from the matrix $M_{14}$}
\label{entries}
\end{figure}

\begin{figure}[ht]
\centering
\scalebox{0.9}{\includegraphics{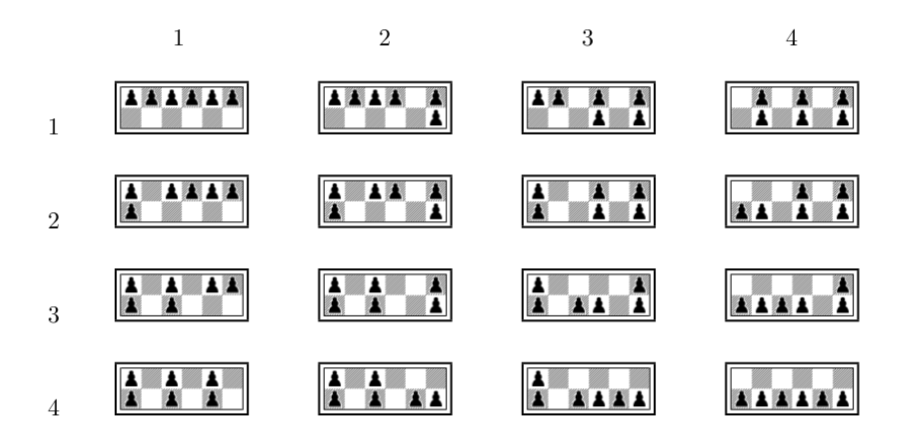}}
\caption{Entries in the matrix $M_6 = (m_{i,j})_{1\leq i,j \leq 4}$}
\label{M6}
\end{figure}

\begin{proposition}\label{uniqueness}
Every nonattacking arrangement of $2m$ pawns on a $2\times 2m$ rectangle appears exactly once in the matrix $M_{2m}$.
\end{proposition}

\begin{proof}
To begin, we show the number of distinct arrangements of pawns on a $2\times 2m$ rectangle is $(m+1)^2$.  For $m=1$, a $2\times 2$ square has the four distinct arrangements shown in Figure~\ref{2times2}, so we induct on $m$.  The leftmost $2\times 2$ square of a $2\times 2m$ rectangle may have Type A, B, C, or D.  First, assume this leftmost square has Type D.  Any maximum independent arrangement of a $2\times 2(m-1)$ rectangle may be appended to the Type D square creating $m^2$ distinct maximum nonattacking arrangements.  Next, if the leftmost square has Type A or C, it must be followed by a square of same type or of Type B.  But in any $2\times 2m$ rectangle, when reading from left to right, as soon as a Type B square is introduced in the strip, all remaining squares to the right must also be of Type B.  Thus any $2\times 2m$ strip beginning with a Type A or Type C square consists of $k$ squares of Type A or C followed by $m-k$ squares of Type B for $1\leq k \leq m$.  Finally there is one possible arrangement beginning with a Type B square.  Thus we have
\begin{equation*}
m^2 +2m+1 = (m+1)^2
\end{equation*}
distinct arrangements as desired.

Further, no arrangement appears more than once in the matrix $M_{2m}$.  We continue to think of the entries of the matrix $M_{2m}$ as a string of $m$ $2\times 2$ squares.  We observe, by construction, as one reads from top to bottom down a column of the matrix, the only actions on these $2\times 2$ squares are:
\begin{enumerate}[i.]
\item Type A squares may be changed to Type D squares.
\item Type B squares may be changed to Type C squares.
\item Any type square may remain fixed.
\end{enumerate}
Similarly, as you read from left to right across a row of the matrix, the only actions are:
\begin{enumerate}[i.]
\item Type A squares may be changed to Type B squares.
\item Type D squares may be changed to Type C squares.
\item Any type square may remain fixed.
\end{enumerate}
Given any two arrangements in distinct positions in the matrix $M_{2m}$, at least one square has changed from the lower-indexed entry to the higher-indexed entry.  If that square was of Type B or D, respectively, it was changed into a Type C square and no action may change it back to a Type B or D square, respectively.  If the square was of Type A, then it was changed to a Type B, C, or D square, but in any case, may not return to Type A.  Because Type C squares cannot be changed, we have a matrix with unique elements whose size is equal to the size of the set, so therefore each independent maximum arrangement of pawns occurs exactly once in $M_{2m}$.
\end{proof}

Now, we define a map from the set of subsets of $m$ rows and $m$ columns of a $2m\times 2m$ matrix into the set of nonattacking arrangements of $2m^2$ pawns.
\begin{definition}
Suppose the rows and columns of a $2m\times 2m$ matrix are indexed by $[2m]$.  Set 
\begin{equation*}
A =\{ C\cup R : C,R\subset [2m] \hbox{ and } |C|=|R|=m\},
\end{equation*}
that is, $A$ is the set of all subsets
consisting of $m$ rows $R=\{r_1, r_2, \ldots, r_{m}\}\subset [2m]$ and $m$ columns $C=\{c_1,c_2, \ldots, c_{m}\}\subset [2m]$.  Let $B$ be the set of all nonattacking arrangements of $2m^2$ pawns on a $2m\times 2m$ chessboard.  Define the map $\Phi: A \longrightarrow B$ as follows:

Given a subset $C\cup R$, assume without loss of generality that $r_1 < r_2< \cdots <r_m$ and $c_1 < c_2 < \cdots < c_m$.  Then set $S$ to be the set of $m$ ordered pairs where 
\begin{equation*}
S=\{(a_i, b_i) : (a_i, b_i) = (r_i-i+1, c_i-i+1) \hbox{ for } 1\leq i \leq m\}.
\end{equation*}
For each ordered pair $(a_i, b_i)$, identify the $2\times 2m$ chessboard arrangement $m_{a_i,b_i}$ from the matrix $M_{2m}$.  Concatenate these strips sequentially so that $m_{a_i,b_i}$ is directly above $m_{a_{i+1}, b_{i+1}}$ for $1\leq i \leq m-1$ to create an arrangement of $2m^2$ pawns on a $2m\times 2m$ chessboard.  This arrangement is the image of the subset $C\cup R$ under $\Phi$.
\end{definition}

\begin{example}
Given $2m=6$, suppose $R=\{1,4,5\}$ and $C=\{2,4,6\}$.  Then $S=\{(1,2), (3,3), (3,4)\}$.  Thus, we concatenate the arrangements $m_{1,2}, m_{3,3}, m_{3,4}$ from Figure~\ref{M6} to get the maximum $6\times 6$ arrangement:
\begin{center}
\includegraphics{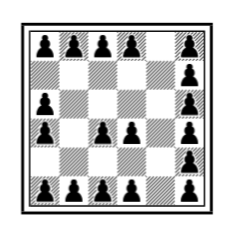}
\end{center}
\end{example}

\begin{example}
Given $2m=8$, suppose $R=\{2,3,4,8\}$ and $C=\{1,6,7,8\}$.  Then $S=\{(2,1),(2,5),(2,5),(5,5)\}$ and we have the following $8\times 8$ arrangement:
\begin{center}
\includegraphics{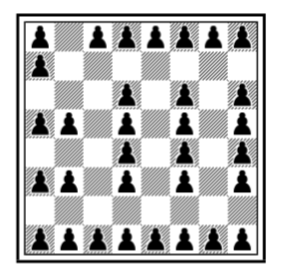}
\end{center}
\end{example}

We check that the arrangements of pawns given by the function $\Phi$ are nonattacking.
\begin{proposition}\label{independent}
Each arrangement of $2m^2$ pawns on a $2m\times 2m$ chessboard in the image $\Phi(A)$ is independent.
\end{proposition}

\begin{proof}
By construction, we know the the pawns may not attack within each $2\times 2m$ rectangular chessboard, so it is left to show that the pawns may not attack from one rectangle to another.  We apply the restrictions on movement along rows and columns noted in the proof of Proposition~\ref{uniqueness}.

 Let $a=m_{i,j}$ and $b=m_{i',j'}$ be any two nonattacking arrangements from the matrix $M_{2m}$ such that $i\leq i'$ and $j\leq j'$.  We assume $a$ lies directly above $b$ in a maximum arrangement of independent pawns on the $2m\times 2m$ chessboard. Divide each $2\times 2m$ rectangle into $2\times 2$ squares and denote a $2\times 2$ square of $a$, or $b$ respectively, at position $k$ where $1\leq k \leq m$ by $A_k$ or $B_k$, respectively.  First, if $A_k$ has Type A, then the pawns in $A_k$ may not attack any pawns in the arrangement $b$.  Next, suppose $A_k$ has Type B, so thus $B_k$ has Type B or Type C.  In either case the pawns in $A_k$ may not attack the pawns in $B_k$.  However the pawn in $A_k$ may also attack the upper left corner of $B_{k+1}$.  Because the square $A_{k+1}$ must also have Type B, we know $B_{k+1}$ has Type B or C.  In either case there is no pawn in the upper left corner, so pawns in $A_k$ may not attack pawns in $B_{k+1}$.  Similarly, if $A_k$ has Type D then $B_i$ also has Type D, and thus no attack is possible.  In this case a pawn in $A_i$ could also attack the upper right corner of $B_{i-1}$.  We see that $A_{i-1}$ also has Type D, so $B_{k-1}$ has Type D and thus no attack is possible from $A_k$ to $B_{k-1}$.  Finally, suppose $A_k$ is of Type C, so pawns in $A_k$ may attack squares $B_{k-1}, B_k$, and $B_{k+1}$.  We know $A_{k-1}$ is of Type C or Type D and $A_{k+1}$ is of Type C.  So we have that $B_{k-1}$ is of Type C  or D, thus not susceptible to an attack from $A_k$.  The squares $B_k$ and $B_{k+1}$ are both of Type C and also have no pawns that may be attacked by pawns in $A_k$.  Finally, we note in any case, if the squares $B_{k-1}$ or $B_{k+1}$ do not exist, then trivially there is no attacking pawn. Therefore, we have shown that any entry weakly to the left or above another entry in $M_{2m}$ may not attack when placed directly above the second entry, and thus have proven the claim.
\end{proof}

We have shown that each subset in $A$ provides exactly one maximum nonattacking arrangement of pawns on a $2m\times 2m$ chessboard, thus $\Phi(A) \subseteq B$.  It is left to show that no other maximum independent arrangements are possible.

\begin{proposition}
Every nonattacking arrangement of $2m^2$ pawns on a $2m\times 2m$ chessboard is the image of a subset $C\cup R \in A$ under the map $\Phi$.
\end{proposition}

\begin{proof}
Any arrangement of $2m^2$ pawns on a $2m\times 2m$ chessboard may be divided into $m$ $2\times 2m$ rectangular boards which correspond to the entries $(m_{a_1, b_1}, \ldots, m_{a_m, b_m})$ in the matrix $M_{2m}$.  For all $i$, as long as $a_i\leq a_{i+1}$ and $b_i \leq b_{i+1}$, then the arrangement is an element of the image $\Phi(A)$.  Suppose to the contrary $a_i > a_{i+1}$ for some $i$.  This implies the arrangement $m_{a_i, b_i}$ is in a lower row in $M_{2m}$ than arrangement $m_{a_{i+1},b_{i+1}}$, but appears directly above $m_{a_{i+1},b_{i+1}}$ in the $2m\times 2m$ arrangement. We apply a similar argument to that used in Proposition~\ref{independent}.  At least one square, say $A_k$ in $m_{a_i,b_i}$ is different from the square in the same position, $B_k$, in $m_{a_{i+1},b_{i+1}}$.  If $A_k$ is of type D, then $B_k$ is of type A or C, hence the pawn in the lower left corner of $A_k$ may attack the pawn in the upper right corner of $B_k$.  If $A_k$ is of Type C, then $B_k$ is of Type A or B, and the pawn in the lower left corner of $A_k$ may attack the pawn in the upper right corner of $B_k$.  Thus $a_i \not> a_{i+1}$.  Similarly, if $b_i > b_{i+1}$, the arrangement $m_{a_i, b_i}$ is in column further to the right in $M_{2m}$ than arrangement $m_{a_{i+1},b_{i+1}}$, but appears directly above $m_{a_{i+1},b_{i+1}}$ in the $2m\times 2m$ arrangement.  Again at least one square, say $A_k$ in $m_{a_i,b_i}$ is different from the square in the same position, $B_k$, in $m_{a_{i+1},b_{i+1}}$.  If $A_k$ of Type B, then $B_k$ is of Type A or D and the pawn in the lower right corner of $A_k$ may attack the pawn in the upper left corner of $B_k$.  Further if $A_k$ is of Type C, then $B_k$ is of Type A or D and the pawn in the lower right corner of $A_k$ may attack the pawn in the upper left corner of $B_k$.  Thus $b_i \not> b_{i+1}$, and we have arrived at the contradiction.
\end{proof}

 Therefore we have the following corollary.

\begin{corollary}
The function $\Phi: A \longrightarrow B$ is a bijection.
\end{corollary}

Hence, because we may choose an $m$-subset of $[2m]$ in ${2m\choose m}$ ways, we have our main result.

\begin{theorem}
The number of maximum nonattacking arrangements of pawns on a $2m\times 2m$ chessboard is ${2m\choose m}^2$.
\end{theorem}

We may generalized this result to maximum independent arrangements of pawns on $2n\times 2m$ rectangles.

\begin{theorem}
The number of maximum nonattacking arrangements of pawns on a $2n\times 2m$ chessboard is ${m+n\choose n}^2$.
\end{theorem}
\begin{proof}
Assume without loss of generality that $n\leq m$.  We may utilize the bijection $\Phi$ from above.  Given a nonattacking arrangement of $2mn$ pawns on a $2n\times 2m$ chessboard, we may divide the arrangement into $n$ rectangles of size $2\times 2m$.  These correspond to $n$ (not necessarily distinct) entries in the matrix $M_{2m}$.  Thus we have a set of indices from the matrix entries
\begin{equation*}
S=\{(a_i, b_i) | 1 \leq a_1 \leq a_2 \leq \cdots \leq a_n\leq m+1 \hbox{ and }1 \leq b_1 \leq b_2 \leq \cdots \leq b_n \leq m+1 \}.
\end{equation*}
Two create distinct column and row entries we have
\begin{equation*}
C \cup R = \{a_1, a_2+1, a_3+2, \ldots, a_n+n-1\} \cup \{b_1, b_2+1, b_2+3, \ldots, b_n+n-1\}.
\end{equation*}
We note the maximum value of elements in $C$ or $R$ is $m+n$, thus $C, R \subset [m+n]$.  Hence we are choosing an $n$-subset of rows from $[m+n]$ and an $n$-subset of columns from $[m+n]$, and the result follows.
\end{proof}


\newcommand{\journal}[6]{{\sc #1,} #2, {\it #3}, {\bf #4}, #6 (#5)}
\newcommand{\book}[4]{{\sc #1,} #2, #3 (#4)}


\begin{thebibliography}{99}

\bibitem{Dudeney}
\book{H.\ E.\ Dudeney}{Amusements in Mathematics}{Edinburgh: Thomas Nelson \& Sons, Limited}{1917}

\bibitem{Larson}
\journal{M.\ Larson}{The Problem of Kings}{Electron. J. Combin.}{2}{1995}{10 pages}

\bibitem{Kitaev_Mansour}
\journal{S.\ Kitaev and T. Mansour}{The Problem of the Pawns}{Annals of Combinatorics}{8}{2004}{pp. 81--91}

\end{thebibliography}
\end{document}